\documentclass{aptpub}
\usepackage{amsmath,amssymb,latexsym}
\usepackage[mathscr]{eucal}
\usepackage{graphicx}
\usepackage{epsfig} 
\usepackage{epstopdf}
\usepackage{cases}
\usepackage{subfig}
\usepackage{color}
\def\para#1{\vskip .4\baselineskip\noindent{\bf #1}}

\numberwithin{equation}{section}

\DeclareMathOperator*{\esssup}{ess\,sup}

\newcommand{\eps}{\varepsilon}

\newcommand{\F}{\mathcal{F}}

\renewcommand{\E}{\mathbb{E}}

\renewcommand{\N}{\mathbb{N}}
\newcommand{\0}{\mathcal{O}}
\newcommand{\PP}{\mathbb{P}}

\newcommand{\R}{\mathbb{R}}

\newcommand{\abs}[1]{\left\vert#1\right\vert}

\numberwithin{equation}{section}

\newcommand{\bed}{\begin{displaymath}}
\newcommand{\eed}{\end{displaymath}}
\newcommand{\bea}{\bed\begin{array}{rl}}
\newcommand{\eea}{\end{array}\eed}

\newcommand{\barray}{\begin{array}{ll}}
\newcommand{\earray}{\end{array}}

\def\disp{\displaystyle}

\newcommand{\1}{\boldsymbol{1}}

\def\bar{\overline}
\def\hat{\widehat}
\def\wdt{\widetilde}

\def\a.s{\text{\;a.s.\;}}

\authornames{D.H. Nguyen, N.N. Nguyen, G. Yin}
\shorttitle{An SPDE Epidemic Model}

\begin{document}

\title{Analysis of A Spatially Inhomogeneous Stochastic Partial Differential Equation Epidemic Model}

\authorone[University of Alabama]{Dang H. Nguyen}
\addressone{University of Alabama, Tuscaloosa, AL
35487, USA, dangnh.maths@gmail.com.}
\authortwo[Wayne State University]{Nhu N. Nguyen}
\addresstwo{Department of Mathematics, Wayne State University, Detroit, MI,
48202, USA, nhu.math.2611@gmail.com.}
\authorthree[Wayne State University]{George Yin}
\addressthree{Department of Mathematics, Wayne State University, Detroit, MI,
48202, USA, gyin@wayne.edu.}

\begin{abstract}
This work proposes and analyzes a family of spatially inhomogeneous epidemic models. This is our first effort to use stochastic partial differential equations (SPDEs) to model  epidemic dynamics with spatial variations and environmental noise. After setting up the problem, existence and uniqueness of solutions of the
 underlying SPDEs are examined. Then definitions of permanence and extinction are given. Certain sufficient conditions are provided for the permanence and extinction. Our hope is that this paper will open up windows for investigation of epidemic models from a new angle.
 \end{abstract}



\section{Introduction}
This work presents an effort of studying stochastic epidemic models, in which spatial in-homogeneity is allowed. The hope is that it will open up a new angle for investigating a large class of epidemic processes. In lieu of the usual stochastic differential equation based formulation considered in the literature, we propose a new class of models by using stochastic partial differential equations.
This effort largely enriches the class of systems and offers great opportunities both mathematically and practically. Meanwhile, it poses greater challenges.

The epidemic models (compartment models), in which the density functions are spatially homogeneous were introduced  in 1927 by Kermack and McKendrick in \cite{Kermach, KM1}.
The main idea is to partition the population into
susceptible, infected, and recovered classes. The dynamics of these classes are given by a system of deterministic differential equations.
One of the classical models takes the form
\begin{equation*}
\begin{cases}
dS(t)=\Big[\Lambda-\mu_SS(t)-\dfrac{\alpha S(t)I(t)}{S(t)+I(t)}\Big]dt\quad t\geq 0,\\[1ex]
dI(t)=\Big[-(\mu_I+r)I(t)+\dfrac{\alpha S(t)I(t)}{S(t)+I(t)}\Big]dt\quad t\geq 0,\\[1ex]
dR(t)=\Big[-\mu_R R(t)+r I(t)\Big]dt\quad t \geq 0,\\
S(0)=S_0\geq 0,\quad I(0)=I_0\geq 0,\quad R(0)=R_0\geq 0,
\end{cases}
\end{equation*}
where $S(t)$, $I(t)$, $R(t)$ are the densities of susceptible, infected, and recovered populations, respectively.
In the above, $\Lambda$ is the recruitment rate of the population; $\mu_S,\mu_I,\mu_R$ are  the
death rates of susceptible, infected and recovered individuals, respectively;  $\alpha$ is
the infection rate and $r$ is the recovery rate.
To simplify the study, it has been noted that
the dynamics of recovered individuals have no effect on the disease transmission dynamics. Thus, following the usual practice,
the recovered individuals are removed from the formulation henceforth. The SIR models are known to be useful and suited for
such diseases as rubella, whooping cough, measles, smallpox, etc.
It has also been well recognized that random effect is  not avoidable and
a population is often
 subject to random disturbances.
Thus,  much effort has also been devoted to the investigation of
 stochastic epidemic models. One popular approach is adding stochastic noise perturbations to the above deterministic models.
In recent years, resurgent
attention has been devoted to
analyzing and designing controls of infectious diseases for host populations; see
\cite{All07,BS,Britton,DDN18, DN17, HDAN, GC,Kor, WZ,Wilkinson} and  references therein.

For the  deterministic models, studying the systems from a dynamic system point of view, certain  threshold-type results have been found. In accordance with the threshold,
  the population tends to the disease-free equilibrium or
  approaches an endemic equilibrium under certain conditions.
It has been a long-time effort to find
  the critical threshold value for the corresponding stochastic systems. A characterization of the systems using critical threshold was done very recently in \cite{DNDY16,DN18, HN18, HNY18}, in which sufficient and almost necessary conditions were obtained using the idea of Lyapunov exponent so that the asymptotic behavior of the system has been completely classified. Such idea can also be found in the work \cite{DNY16, NY17} for related problems.

From another angle,
it has been widely recognized that there should be spatial dependence in the model, which will better reflect the spacial variations.
In the spatially inhomogeneous case, the epidemic reaction-diffusion system takes the form
\begin{equation}
\begin{cases}
\disp {\partial \over \partial t}
S(t,x)=k_1\Delta S(t,x) +\Lambda(x)-\mu_1(x)S(t,x)-\dfrac{\alpha(x) S(t,x)I(t,x)}{S(t,x)+I(t,x)}\quad\text{in }\R^+\times\0,\\
\disp {\partial \over \partial t}
I(t,x)=k_2\Delta I(t,x)-\mu_2(x) I(t,x) + \dfrac{\alpha(x) S(t,x)I(t,x)}{S(t,x)+I(t,x)} \quad\text{in }\R^+\times\0,\\[1ex]
\partial_{\nu}S(t,x)=\partial_{\nu}I(t,x)=0\quad\quad\quad\quad\;\text{in }\R^+\times\partial\0,\\[1ex]
S(x,0)=S_0(x),I(x,0)=I_0(x)\quad\;\;\text{in }\0,
\end{cases}
\end{equation}
where $\Delta $ is the Laplacian with respect to the spatial variable, $\0$ is a bounded domain with $C^2$ boundary
of $\R^l$ ($l\geq 1$), $\partial_{\nu}S$ denotes the directional derivative with the $\nu$ being the outer normal direction on $\partial \0$, and $k_1$ and $k_2$ are positive constants representing the diffusion rates of the susceptible and infected population densities, respectively. In addition, $\Lambda(x),\mu_1(x),\mu_2(x),\alpha(x)\in C^2(\0)$ are non-negative functions.
Recently, the epidemic reaction-diffusion models have been studied in \cite{Allen,Ducrot,Peng2009,Peng,Zhang} and the references therein. In \cite{WZ}, some results were given for a general epidemic model with reaction-diffusion in terms of basic reproduction numbers.
The above models are all noise free.
However,  random noise perturbations in
the environment often inevitably appear.
Therefore,
a more suitable description requires to
 consider stochastic epidemic diffusive models.
 Taking this into consideration, we propose a spatially non-homogeneous model using a system of stochastic partial differential equations given by
\begin{equation}\label{eq}
\begin{cases}
dS(t,x)=\Big[k_1\Delta S(t,x) +\Lambda(x)-\mu_1(x)S(t,x)-\dfrac{\alpha(x) S(t,x)I(t,x)}{S(t,x)+I(t,x)}\Big]dt\\
\hspace{7.5cm}+ S(t,x)dW_1(t,x)\quad\text{in }\R^+\times\0,\\
dI(t,x)=\Big[k_2\Delta I(t,x)-\mu_2(x) I(t,x)+\dfrac{\alpha(x) S(t,x)I(t,x)}{S(t,x)+I(t,x)}\Big]dt
\\\hspace{7.5cm}+I(t,x)dW_2(t,x)\quad\text{in }\R^+\times\0,\\
\partial_{\nu}S(t,x)=\partial_{\nu}I(t,x)=0\quad\quad\quad\quad\quad\quad\quad\text{in }\R^+\times\partial\0,\\
S(x,0)=S_0(x),I(x,0)=I_0(x)\quad\quad\quad\quad\;\text{in }\0,
\end{cases}
\end{equation}
where $W_1(t,x)$ and $W_2(t,x)$ are $L^2(\0,\R)$-value Wiener processes, which present the noises in both time and space.
We refer the readers to \cite{Prato} for more details on the $L^2(\0,\R)$-value Winner process.

Because this is our first work in this direction, we have to settle a number of issues. First, we establish the existence and uniqueness of solutions in the sense of mild solution of the stochastic partial differential equations. Moreover, we examine some long-term behavior of the solutions.  These are the main objectives of the current work.

The rest of the paper is arranged as follows.
Section \ref{sec:pre} gives some preliminary results and also formulates the problem that we wish to study.
Section \ref{sec:exist} establishes the existence and uniqueness of the solution of the stochastic partial differential equations. Section \ref{sec:longtime} provides sufficient conditions for the extinction and permanence while Section \ref{example} provides an example. Finally, Section \ref{sec:con} concludes the paper with some further remarks.

\section{Preliminary and Formulation}\label{sec:pre}
Let $\0$ be a bounded domain in $\R^l$ (with $l\geq 1$) having $C^2$ boundary and $H:=L^2(\0;\R)$ be the separable Hilbert space, endowed
with
the scalar product
$$\langle  u, v \rangle_H:=\int_{\0} u(x) v(x)dx,$$
and the corresponding norm $\abs{ u}_H = \sqrt {\langle  u,  u \rangle_H}$. We will say $ u\geq 0$ if $ u(x)\geq 0$ almost everywhere in $\0$. Moreover, we denote by $L^2(\0,\R^2)$ the space of all functions $u(x)=\big(u_1(x),u_2(x)\big)$ where $ u_1,u_2\in L^2(\0,\R)$, on which
the inner product is defined as
\begin{align*}
\langle u,v \rangle_{L^2(\0,\R^2)}&:=\int_\0 \big\langle u(x),v(x) \big\rangle_{\R^2}dx=\int_{\0}\big(u_1(x)v_1(x)+u_2(x)v_2(x)\big)dx
\\&=\langle u_1,v_1 \rangle_{L^2(\0,\R)}+\langle u_2,v_2 \rangle_{L^2(\0,\R)},
\end{align*}
for all $u,v\in L^2(\0,\R^2)$.
Note that $L^2(\0,\R^2)$ is  a separable Hilbert space. In what follows, we use $u$ to denote a function that is either real-valued or an $\R^2$-valued. It will be clear from the context.
Denote by $E$ the Banach space $C(\bar \0;\R)$ endowed with the sup-norm
$$\abs{u}_E:=\sup_{x\in\bar \0}\abs{u(x)}.$$
Let $\big(\Omega, \F,\{\F_t\}_{t\geq 0},\PP\big)$ be a complete probability space and $L^p\big(\Omega;C([0,t],C(\bar \0,\R^2))\big)$ be the space of all predictable $C(\bar\0,\R^2)$-valued processes $u$ in $C\big([0,t],C(\bar \0,\R^2)\big)$, $\PP$-a.s. with the norm $L_{t,p} $ as follows
$$\abs{u}^p_{L_{t,p}}:=\E \sup_{s\in [0,t]}\abs{u(s)}^p_{C(\bar\0,\R^2)},$$
where
$$\abs{u}_{C(\bar\0,\R^2)}=\Big(\sum_{i=1}^2\sup_{x\in\bar \0}\abs{u_i(x)}^2\Big)^\frac 12\quad\text{if}\quad u=(u_1,u_2)\in C(\bar\0,\R^2).$$
For $\eps>0,p\geq 1$, denote by $W^{\eps,p}(\0,\R^2)$ the Sobolev-Slobodeckij space (the Sobolev space with non-integer exponent) endowed with the norm
$$\abs{u}_{\eps,p}:=\abs{u}_{L^p(\0,\R^2)}+\sum_{i=1}^2\int_{\0\times\0}\dfrac{\abs{u_i(x)-u_i(y)}^p}{\abs{x-y}^{\eps p+l}}dxdy.$$

Assume that $B_{k,1}(t)$ and $B_{k,2}(t)$ with $k=1,2,\dots$,
are independent $\{\F_t\}_{t\geq 0}$-adapted one-dimensional Wiener processes. Now, fix an orthonormal basis $\{e_k\}_{k=1}^{\infty}$ in $H$ and
assume
that this sequence
is uniformly bounded in $L^{\infty}(\0,\R)$, i.e.,
$$C_0:=\sup_{k\in\N}\abs{e_k}_{L^{\infty}(\0,\R)}=\sup_{k\in\N}\esssup_{x\in\0} e_k(x)<\infty.$$
We define the
 infinite dimensional Wiener processes $W_i(t)$,
 which are driving noises in
 equation \eqref{eq} as follows
$$\displaystyle W_i(t)=\sum_{k=1}^{\infty}\sqrt {a_{k,i}}B_{k,i}(t)e_k,\quad i=1,2,$$
where $\{a_{k,i}\}_{k=1}^{\infty}$ are sequences of non-negative real numbers satisfying
\begin{equation}\label{condition}
a_i:=\sum_{k=1}^{\infty}a_{k,i}<\infty,\quad i=1,2.
\end{equation}
Let $A_1$ and $A_2$ be Neumann realizations of $k_1\Delta$ and $k_2\Delta$ in $H$, respectively, i.e.,
$$D(A_i)=\Big\{u\in H\big| \Delta u\in H\;\text{and}\;\partial_\nu u=0\;\text{on}\;\partial \0\Big\},$$
$$A_i u=k_i\Delta u, \;\;u\in D(A_i),$$
where the Laplace operator in the above definition is understood in the distribution sense. Then, $A_1$ and $A_2$ are infinitesimal generators of analytic semi-groups $e^{tA_1}$ and $e^{tA_2}$ with corresponding Neumann heat kernels, denoted by $p_\0^{N,1}(t,x,y),p_\0^{N,2}(t,x,y)$, i.e.,
$$(e^{tA_i}u)(x)=\int_\0 p_\0^{N,i}(t,x,y)u(y)dy,\;i=1,2,$$
 respectively. In addition, if we denote $A:=(A_1,A_2)$, the operator defined in $L^2(\0,\R^2)$ by $Au:=(A_1u_1,A_2u_2)$ for $u=(u_1,u_2)\in L^2(\0,\R^2)$, then it generates an analytic semigroup $e^{tA}$ with $e^{tA}u=(e^{tA_1}u_1,e^{tA_2}u_2).$
In \cite[Theorem 1.4.1]{Davies}, it is proved that the space $L^1(\0,\R^2 )\cap L^\infty(\0,\R^2)$ is
invariant under $e^{tA}$, so that $e^{tA}$ may be extended to a non-negative one-parameter semigroup $e^{tA(p)}$ on $L^p(\0;R^2 )$, for all $1\leq p\leq\infty$.
All these semi-groups are strongly continuous and consistent in the sense that $e^{tA(p)}u=e^{tA(q)}u$ for any $u\in L^p(\0,\R^2)\cap L^q(\0,\R^2)$ (see \cite{Cerrai-book}). So, we will suppress
the superscript $p$ and denote them by $e^{tA}$ whenever there is no confusion.
Moreover, if we consider the part $A_i^E$ of $A_i$ in the space of continuous functions $E$, it generates an analytic semi-group (see \cite[Chapter 2]{Arendt}), which has no dense domain in general. However, since we have assumed that $\0$ has $C^2$ boundary, in our boundary condition, $A_i^E$ has dense domain in $E$ (see \cite[Appendix A.5.2]{Prato}) and hence, this analytic semi-group is strongly continuous. Finally, we recall some well-known properties of the operators $A_i$ and analytic semi-groups $e^{tA_i}$ for $i=1,2$ as follows. For further details,
we refer the reader to the monographs \cite{Arendt,Davies,Ouhabaz} and the references therein.
\begin{itemize}
\item $\forall u\in H$ then $\int_ 0^t e^{sA_i}uds\in D(A_i)$ and $A_i(\int_0^t e^{sA_i}uds)=e^{tA_i}u-u$.
\item
By Green's identity, it can be proved that $A_i$ is symmetric, that $A_i$ is self-adjoint in $H$,
and that $\forall u\in D(A_i)$, $\int_\0 (A_iu)(x)dx=0$.
\item For any $t>0,x,y\in\0$,
$$0\leq p_\0^{N,i}(t,x,y)\leq c_1 (t\wedge 1)^{-\frac l2}e^{-c_2\frac{\abs{x-y}^2}{t}},$$
for some constant $c_1,c_2$, which depends on $\0$, but is independent of $u,t$.

\item The semigroup $e^{tA}$ satisfies the fowling properties
\begin{equation}\label{propertyA2}
\begin{aligned}
\abs{e^{tA}u}_{L^\infty(\0,\R^2)}\leq c\abs{u}_{L^\infty(\0,\R^2)}\;\text{and}\;\abs{e^{tA}u}_{C(\bar\0,\R^2)}\leq c\abs{u}_{C(\bar\0,\R^2)},
\end{aligned}
\end{equation}
for some constant $c$, which depends on $\0$, but is independent of $u,t$.

\item For any $t,\eps>0$, $p\geq 1$, the semigroup $e^{tA}$ maps $L^p(\0,\R^2)$ into $W^{\eps,p}(\0,\R^2)$
 and $\forall u\in L^p(\0,\R^2)$
\begin{equation}\label{propertyA}
\abs{e^{tA}u}_{\eps,p}\leq c(t\wedge 1)^{-\eps/2}\abs{u}_{L^p(\0,\R^2)},
\end{equation}
for some constant $c$ independent of $u,t$.
\end{itemize}

Now, we rewrite equation \eqref{eq} as the stochastic differential equation in an infinite dimension space
\begin{equation}\label{eq1}
\begin{cases}
dS(t)=\Big[A_1 S(t) +\Lambda-\mu_1S(t)-\dfrac{\alpha S(t)I(t)}{S(t)+I(t)}\Big]dt + S(t)dW_1(t),\\
dI(t)=\Big[A_2 I(t)-\mu_2I(t) + \dfrac{\alpha S(t)I(t)}{S(t)+I(t)}\Big]dt+I(t)dW_2(t),\\
S(0)=S_0,I(0)=I_0.
\end{cases}
\end{equation}
As usual, we say that $\big(S(t),I(t)\big)$ is a mild solution to \eqref{eq1}, if
\begin{equation}\label{so}
\begin{cases}
\displaystyle S(t)=e^{tA_1}S_0+\int_0^t e^{(t-s)A_1}\Big(\Lambda-\mu_1S(s)-\dfrac{\alpha S(s)I(s)}{S(s)+I(s)}\Big)ds+W_S(t),\\
\displaystyle I(t)=e^{tA_2}I_0+\int_0^t e^{(t-s)A_2}\Big(-\mu_2 I(s)+\dfrac{\alpha S(s)I(s)}{S(s)+I(s)}\Big)ds+W_I(t),
\end{cases}
\end{equation}
where
$$W_S(t)=\int_0^t e^{(t-s)A_1}S(s)dW_1(s)\;\text{and}\;W_I(t)=\int_0^t e^{(t-s)A_2}I(s)dW_2(s),$$
or in the vector form
\begin{equation}\label{so-vec}
\displaystyle Z(t)=e^{tA}Z_0+\int_0^t e^{(t-s)A}F\big(Z(s)\big)ds+\int_0^t e^{(t-s)A}Z(s)dW(s),
\end{equation}
where $ Z=(S,I)$, $F(Z)=\big(F_1(Z),F_2(Z)\big):=\Big(\Lambda-\mu_1S-\dfrac{\alpha SI}{S+I},-\mu_2I+\dfrac{\alpha SI}{S+I}\Big)$, and
$$e^{(t-s)A}Z(s)dW(s):=\big(e^{(t-s)A_1}S(s)dW_1(s)\;,\; e^{(t-s)A_2}I(s)dW_2(s)\big).$$

 Because we are modeling the SIR epidemic systems, we are only interested in the positive ($\geq 0$) solutions. Therefore, we define a ``{\it positive mild solution}" of \eqref{eq1} as a mild solution $S(t,x),I(t,x)$ such that $S(t,x),I(t,x)\geq 0$, almost everywhere $x\in\0$, for all $t\geq 0$. Moreover, to have the term $\frac{si}{s+i}$ well defined, we assume that it is equal 0 whenever either $s=0$ or $i=0$.

\begin{rem}
The  integrals on the right-hand side of \eqref{so} are understood as  Bochner integrals (in the Banach space $H$) while
$W_S(t)$ and $W_I(t)$
are the stochastic integrals (stochastic convolutions).
The $S(s)$ (resp. $I(s)$) in the stochastic integrals is understood as multiplication operator, i.e.,
$$S(s)(u)=S(s)u,\quad \forall u\in H.$$
The stochastic integral $\int_0^t e^{(t-s)A_i}U(s)dW_i(s)$ (see \cite[Chapter 4]{Prato} for more details on  stochastic integrals) is well-define if the process $U(s)$ satisfies that
$$\int_0^t \sum_{k=1}^{\infty}a_{k,i}\abs{e^{(t-s)A_i} U(s) e_k}^2_Hds<\infty.$$
Finally, in the vector form, to simplify notation, we do not write the vectors in the column form. However, the calculations involving vectors are understood as in the usual sense.
\end{rem}

To investigate the epidemic models, an important question is  whether the infected individual will die out in the long time. That is, the consideration of extinction or permanence. Since the mild solution is used, let us introduce the definitions in the weak sense as follows.

\begin{defn}{\rm
A population
with density
$u(t,x)$ is said to be extinct in the mean
if
$$\limsup_{t\to\infty}\dfrac 1t\int_0^t\E\int_\0 u(s,x)dxds=0,$$
and that is said to be permanent in the mean if there exists a positive number $R_I$, is independent of initial conditions of population, such that
$$\liminf_{t\to\infty}\dfrac 1t\int_0^t\Big(\E\int_\0 \big(u^2(s,x)\wedge 1\big)dx\Big)^{\frac 12}ds\geq R_I.$$
}\end{defn}

\begin{rem}
It is well known that it is fairly difficult to confirm the existence of
 strong solutions for stochastic partial differential equations (even weak solution); see
 \cite[Section 6.1]{Prato}.
As an alternative, we shall use the notion of mild solutions. Hence, the convergence in our situation is in the weak sense. Note however, in the deterministic case, in \cite{Allen,Ducrot,Peng2009,Peng,Zhang}, the authors obtained strong solutions of the deterministic reaction-diffusion epidemic models and the convergence is taken in a
  space such as $L^{\infty}$, $E$, or a Sobolev space.
  In what follows, for convenience, we often suppress the phrase ``in the mean'' when we refer to extinction and permanence, because we are mainly working with mild solutions.
\end{rem}

\section{Existence and Uniqueness of the Positive Mild Solution}\label{sec:exist}
In this section, we shall prove the existence and uniqueness of the positive mild solution of the system as well as its continuous dependence on initial conditions. In what follows, without loss of the generality we can assume $\abs{\0}=1$, where $\abs{\0}$ is the volume of bounded domain $\0$ in $\R^l$ and the initial values are non-random for  simplicity.

\begin{thm}\label{existence}
For any initial data $0\leq S_0,I_0 \in E$, there exists a unique positive mild solution $\big(S(t),I(t)\big)$ of \eqref{eq1} belongs to
$L^p\big(\Omega; C([0,T],C(\bar\0,\R^2))\big)$
for any $T>0, p \geq 1.$
Moreover, this solution depends continuously on the initial data.
\end{thm}

\begin{proof}
In this proof, the letter $c$  denotes a positive constant whose value may change in
different occurrences. We will write the dependence of constant on parameters explicitly if it is essential.
First, we rewrite the coefficients by defining $f$ and $f^*$
as follows:
$$f(x,s,i)=\Big(\Lambda(x)-\mu_1(x)s-\dfrac{\alpha(x) si}{s+i},-\mu_2(x)i+\dfrac{\alpha(x) si}{s+i}\Big),\;x\in\0,\;(s,i)\in \R^2,$$
and
$$f^*(x,s,i)=f\big(x,s\vee 0,i\vee 0\big).$$
Writing $z=(s,i)$, by noting that as our assumption, the term $\dfrac {si}{s+i}$ will be equal to $0$ whenever either $s=0$ or $i=0$, it is easy to see that $f^*(x,\cdot,\cdot):\R^2 \mapsto
\R^2$ is Lipschitz continuous, uniformly in $x\in\0$ so that the composition operator $F^*(z)$ associated with $f^*$, i.e.,
$$F^*(z)(x)=\big(F^*_1(z)(x),F_{2}^*(z)(x)\big):=f^*\big(x,z(x)\big),\quad x\in \0,$$
is Lipschitz continuous, both in $L^2(\0, \R^2)$ and $C(\bar \0,\R^2)$.
Now, we consider the following problem
\begin{equation}\label{Z_n}
dZ^*(t)=\big[AZ^*(t)+F^*\big(Z^*(t)\big)\big]dt+\big(Z^*(t)\vee 0\big)dW(t), \quad Z^*(0)=Z_0=(S_0,I_0),
\end{equation}
where $Z^*(t)=\big(S^*(t),I^*(t)\big)$ and $Z^*(t)\vee 0$ is defined by $$\big(Z^*(t)\vee 0\big)(x)=\big(S^*(t,x)\vee 0,I^*(t,x)\vee 0\big).$$
For any
$$u(t,x)=\big(u_1(t,x),u_2(t,x)\big)\in L^p\big(\Omega; C([0,T],C(\bar\0,\R^2))\big),$$
 consider the mapping
$$\gamma(u)(t):=e^{tA}Z_0+\int_0^t e^{(t-s)A}F^*\big(u(s)\big)ds+\varphi(u)(t),$$
where
\begin{align*}
\varphi(u)(t)&:=\int_0^t e^{(t-s)A}\big(u(s)\vee 0\big)dW(s)
\\&:=\Big(\int_0^t e^{(t-s)A_1}\big(u_1(s)\vee 0\big)dW_1(s),\int_0^t e^{(t-s)A_2}\big(u_2(s)\vee 0\big)dW_2(s)\Big).
\end{align*}
 We will prove that $\gamma$ is a contraction mapping in $L^p\big(\Omega; C([0,T_0],C(\bar\0,\R^2))\big),$ for some $T_0>0,$ and any $p\geq p_0$ for some $p_0$.

\begin{lem}\label{varphi}
There exists $p_0$ such that for any $p\geq p_0$, the mapping $\varphi$
$$\text{maps  }L^p\big(\Omega; C([0,t],C(\bar\0,\R^2))\big)\text{ into itself},$$
and for any $u=(u_1,u_2),v=(v_1,v_2)\in L^p\big(\Omega; C([0,t],C(\bar\0,\R^2))\big)$
\begin{equation}\label{estimatelambda}
\abs{\varphi (u)-\varphi (v)}_{L_{t,p}}\leq c_p(t)\abs{u-v}_{L_{t,p}},
\end{equation}
where $c_p(t)$ is some constant satisfying $c_p(t)\downarrow 0$ as $t\downarrow 0.$
\end{lem}

\begin{proof}
Let $p_0$ be sufficiently large to satisfy that for any $p\geq p_0$, we can choose simultaneously $\beta,\eps>0$ such that
\begin{equation*}
\frac 1p<\beta<\frac 12\quad\text{and}\quad \frac lp<\eps<2\big(\beta-\frac 1p\big).
\end{equation*}
Now, for any fixed $p\geq p_0$, let $\beta,\eps$ be chosen as above.
By using a factorization argument (see e.g., \cite[Theorem 8.3]{Prato}), we have
\begin{equation*}
\varphi(u)(t)-\varphi(v)(t)=\dfrac{\sin \pi \beta}{\pi}\int_0^t (t-s)^{\beta-1}e^{(t-s)A}Y_\beta(u,v)(s)ds,
\end{equation*}
where
$$Y_\beta(u,v)(s)=\int_0^s (s-r)^{-\beta}e^{(s-r)A}\Big(u(r)\vee 0-v(r)\vee 0\Big)dW(r).$$
If
$$\int_0^t \abs{Y_\beta(u,v)(s)}_{L^p(\0,\R^2)}^pds<\infty,\hbox{ a.s.},$$
then it is easily seen from the properties \eqref{propertyA} of semi-group $e^{tA}$  and H\"older's inequality that
\begin{equation}\label{lambda}
\begin{aligned}
&\abs{\varphi(u)(t)-\varphi(v)(t)}_{\eps,p}\\
&\quad \leq c_\beta\int_0^t (t-s)^{\beta-1}\big((t-s)\wedge 1\big)^{-\eps/2}\abs{Y_\beta(u,v)}_{L^p(\0,\R^2)}ds
\\
&\quad \leq c_{\beta,p}(t)\Big(\int_0^t \big((t-s)\wedge 1\big)^{\frac p{p-1}(\beta-\eps/2-1)}ds\Big)^{\frac{p-1}p}\Big(\int_0^t \abs{Y_\beta(u,v)(s)}_{L^p(\0,\R^2)}^pds\Big)^{\frac 1p}\\
&\quad\leq c_{\beta,p}(t)\Big(\int_0^t \abs{Y_\beta(u,v)(s)}_{L^p(\0,\R^2)}^pds\Big)^{\frac 1p},\a.s ,
\end{aligned}
\end{equation}
where $c_{\beta,p}(t)$ is some positive constant, satisfies $c_{\beta,p}(t)\downarrow 0$ as $t\downarrow 0$. Rewriting $Y_\beta(u,v)(s)=\big(Y_{1\beta}(u,v)(s),Y_{2\beta}(u,v)(s)\big)$, where
$$Y_{i\beta}(u,v)(s):=\int_0^s (s-r)^{-\beta}e^{(s-r)A_i}\big(u_i(r)\vee 0-v_i(r)\vee 0\big)dW_i(r),\;i=1,2.$$
Therefore,
applying the Burkholder inequality, we obtain that for all $s\in [0,t]$, almost every $x\in\0$,
\begin{equation*}
\begin{aligned}
&\!\!\!\E \abs{Y_{i\beta}(u,v)(s,x)}^p\leq c_p\E\Big[\int_0^s(s-r)^{-2\beta}\sum_{k=1}^{\infty}a_{k,i}\abs{M_i(s,r,k,x)}^2dr\Big]^{\frac p2}.
\end{aligned}
\end{equation*}
where
$$M_i(s,r,k)=e^{(s-r)A_i}\big(u_i(r)\vee 0-v_i(r)\vee 0\big)e_k.$$
In above, we used the notations
$$
Y_{i\beta}(u,v)(s,x):=Y_{i\beta}(u,v)(s)(x),\quad M_i(s,r,k,x):=M_i(s,r,k)(x),\;i=1,2.
$$
As a consequence,
\begin{equation}\label{Ybeta}
\begin{aligned}
\E \int_0^t &\abs{Y_\beta(u,v)(s)}_{L^p(\0,\R^2)}^pds\\
&\leq c_p(t)\E\int_0^t\int_\0 \Big(\abs{Y_{1\beta}(u,v)(s,x)}^p+\abs{Y_{2\beta}(u,v)(s,x)}^p\Big)dxds
\\&\leq c_p(t) \int_0^t\E\Big(\int_0^s (s-r)^{-2\beta}(a_1+a_2)\sup_{k\in\N}\abs{M(s,r,k)}_{L^{\infty}(\0,\R^2)}^2dr\Big)^\frac p2ds,
\end{aligned}
\end{equation}
where $M(s,r,k):=\big(M_1(s,r,k),M_2(s,r,k)\big)$ and $a_1,a_2$ are defined in \eqref{condition}. Moreover, since the uniformly boundedness property of $\{e_k\}_{k=1}^\infty$ and \eqref{propertyA2}, we have
\begin{equation}\label{Mr}
\sup_{k\in\N}\abs{M(s,r,k)}_{L^\infty(\0,\R^2)}\leq c\abs{u(r)-v(r)}_{C(\bar\0,\R^2)},
\end{equation}
for some constant $c$ independent of $s,r,u,v$.
Combining \eqref{Ybeta} and \eqref{Mr} implies that
\begin{equation}\label{boundedYbeta}
\begin{aligned}
\E \int_0^t &\abs{Y_\beta(u,v)(s)}_{L^p(\0,\R^2)}^pds
\\&\leq c_p(t)\int_0^t\E\sup_{r\in [0,s]}\abs{u(r)-v(r)}_{C(\bar\0,\R^2)}^p\Big(\int_0^s (s-r)^{-2\beta}dr\Big)^\frac p2ds
\\&\leq c_{\beta,p}(t)\int_0^t\E\sup_{r\in [0,s]}\abs{u(r)-v(r)}_{C(\bar\0,\R^2)}^pds\leq c_{\beta,p}(t)\abs{u-v}^p_{L_{t,p}}<\infty,
\end{aligned}
\end{equation}
where $c_{\beta,p}(t)$ is some positive constant and satisfies $c_{\beta,p}(t)\downarrow 0$ as $t\downarrow 0$. Therefore, the inequality \eqref{lambda} holds and as a consequence, $\varphi(u)(t)-\varphi(v)(t)\in W^{\eps,p}(\0,\R^2)$. Since $\eps >l/p$, the Sobolev embedding theorem implies that $\varphi(u)(t)-\varphi(v)(t)\in C(\bar\0,\R^2).$
Finally, \eqref{lambda} and \eqref{boundedYbeta} imply that
\begin{equation*}
\abs{\varphi(u)-\varphi(v)}_{L_{t,p}}\leq c_p(t)\abs{u-v}_{L_{t,p}}
\end{equation*}
for some constant $c_p(t)$, satisfying $c_p(t)\downarrow 0$ as $t\downarrow 0.$ The Lemma is proved.
\end{proof}

Therefore, for $p\geq p_0$, with sufficiently large $p_0$, $\gamma$ maps $L^p\big(\Omega; C([0,t],C(\bar\0,\R^2))\big)$ into itself. Moreover, by using \eqref{propertyA2}
 and Lipschitz continuity of $F^*$,
 we have
\begin{equation}\label{estimateF}
\begin{aligned}
\int_0^t&\abs{e^{(t-s)A}\big[F^*\big(u(s))-F^*(v(s)\big)\big]}_{C(\bar\0,\R^2)}^pds\leq c\int_0^t \abs{\big(u(s)-v(s)\big)}^p_{C(\bar\0,\R^2)}ds
\\&\leq c\int_0^t \sup_{r\in[0,s]}\abs{\big(u(r)-v(r)\big)}_{C(\bar\0,\R^2)}^pds\leq ct\sup_{s\in [0,t]}\abs{u(s)-v(s)}_{C(\bar\0,\R^2)}^p.
\end{aligned}
\end{equation}
Hence, \eqref{estimatelambda} and \eqref{estimateF} imply that
\begin{equation*}
\abs{\gamma (u)-\gamma (v)}_{L_{t,p}}\leq c_p(t)\abs{u-v}_{L_{t,p}},
\end{equation*}
where $c_p(t)$ is some constant depending on $p,t$ and satisfying $c_p(t)\downarrow 0$ as $t\downarrow 0$.
Therefore, for some $T_0$  sufficient small,
 $\gamma$ is a contraction mapping in $L^p\big(\Omega; C([0,T_0],C(\bar\0,\R^2))\big).$ By a fixed point argument we can conclude that equation \eqref{Z_n} admits a unique mild solution in $L^p\big(\Omega; C([0,T_0],C(\bar\0,\R^2))\big).$ Thus, by repeating the above argument in each finite time interval $[kT_0,(k+1)T_0]$, for any $T>0,p\geq p_0$ the equation \eqref{Z_n} admits a unique mild solution $Z^*(t)=\big(S^*(t),I^*(t)\big)$ in $L^p\big(\Omega; C([0,T],C(\bar\0,\R^2))\big).$
We proceed to prove the positivity of $S^*(t),I^*(t)$.

\begin{lem}\label{pos}
Let $\big(S^*(t),I^*(t)\big)$ be the unique mild solution of \eqref{Z_n}. Then $\forall t\in [0,T]$, $S^*(t),I^*(t)\geq 0$ a.s.
\end{lem}

\begin{proof}
Equivalently, $\big(S^*(t),I^*(t)\big)$ is the mild solution of the equation
\begin{equation}\label{S^*,I^*}
\begin{cases}
dS^*(t)=\big[A_1 S^*(t) + F_{1}\big(S^*(t)\vee 0,I^*(t)\vee 0\big)\big]dt + \big(S^*(t)\vee 0\big)dW_1(t),\\
dI^*(t)=\big[A_2 I^*(t) + F_{2}\big(S^*(t)\vee 0,I^*(t)\vee 0\big)\big]dt+\big(I^*(t)\vee 0\big)dW_2(t),\\
S^*(0)=S_0,I^*(0)=I_0.
\end{cases}
\end{equation}
For $i=1,2$, let $\lambda_i\in \rho(A_i)$ be the resolvent set of $A_i$ and $R_i(\lambda_i):=\lambda_i R_i(\lambda_i,A_i)$, with $R_i(\lambda_i,A_i)$ being the resolvent of $A_i$. For each small $\eps>0$, $\lambda=(\lambda_1,\lambda_2)\in \rho(A_1)\times\rho(A_2)$, by \cite[Proposition 1.3.6]{Kailiu}, there exists a unique strong solution $S_{\lambda,\eps}(t,x),I_{\lambda,\eps}(t,x)$ of the equation
\begin{equation}\label{4.2}
\begin{cases}
\displaystyle
dS_{\lambda,\eps}(t)=\Big[A_1S_{\lambda,\eps}(t)+R_1(\lambda_1)F_{1}\big(\eps\Phi(\eps^{-1} S_{\lambda,\eps}(t)),\eps\Phi(\eps^{-1} I_{\lambda,\eps}(t))\big)\Big]dt
\\\quad\quad\quad\quad\quad\quad+R_1(\lambda_1)\eps\Phi\big(\eps^{-1} S_{\lambda,\eps}(t)\big)dW_1(t),\\
\displaystyle
dI_{\lambda,\eps}(t)=\Big[A_2 I_{\lambda,\eps}(t)+ R_2(\lambda_2)F_{2}\big(\eps\Phi(\eps^{-1} S_{\lambda,\eps}(t)),\eps\Phi(\eps^{-1} I_{\lambda,\eps}(t))\big)\Big]dt
\\ \quad\quad\quad\quad\quad\quad+R_2(\lambda_2)\eps\Phi\big(\eps^{-1} I_{\lambda,\eps}(t)\big)dW_2(t),\\
S_{\lambda,\eps}(0)=R_1(\lambda_1)S_0,\quad I_{\lambda,\eps}(0)=R_2(\lambda_2)I_0,
\end{cases}
\end{equation}
where
\[
\Phi(\xi)=\begin{cases}0 \quad\text{if}\;\xi\leq 0,\\
3\xi^5-8\xi^4+6\xi^3 \quad\text{if}\;0<\xi< 1,\\
\xi \quad\text{if}\;\xi\geq1,
\end{cases}
\]
satisfying
\begin{equation*}
\begin{cases}
\Phi\in C^2(\R),\\
\eps\Phi(\eps^{-1}\xi)\to \xi\vee 0\quad \text{as}\;\eps\to 0.
\end{cases}
\end{equation*}
Combined with the convergence property in \cite[Proposition 1.3.6]{Kailiu}, we obtain that
$\big(S_{\lambda(k),\eps}(t)$, $I_{\lambda(k),\eps}(t)\big) \to \big(S^{*}(t),I^{*}(t)\big)$ in $L^p\big(\Omega; C([0,T],L^2(\0,\R^2))\big)$ for some sequence $\{\lambda(k)\}_{k=1}^{\infty}\subset \rho(A_1)\times\rho(A_2)$ and $\eps\to 0$.

Now, let
\[
g(\xi)=\begin{cases}\xi^2-\dfrac 16\quad\quad\;\;\;\text{if}\;\xi\leq -1,\\
-\dfrac {\xi^4}2-\dfrac{4\xi^3}{3}\quad\text{if}\;-1<\xi< 0,\\
0 \quad\quad\quad\quad\quad\;\text{if}\;\xi\geq 0.
\end{cases}
\]
Then $g'(\xi)\leq 0\;\forall \xi$ and $g''(\xi)\geq 0\;\forall \xi$. Hence,  we are to compute $d_t\big(\int_\0 g(I_{\lambda,\eps}(t,x)dx)\big)$. Since the fact $g'(\xi)\Phi(\xi)=g''(\xi)\Phi(\xi)=0\;\forall \xi$, by It\^o's Lemma \cite[Theorem 3.8]{Curtain}, we get
\begin{equation*}
\begin{aligned}
\int_\0 g\big(I_{\lambda,\eps}(t,x)\big)dx&=k_2\int_0^t\int_\0 g'\big(I_{\lambda,\eps}(s,x)\big)\Delta I_{\lambda,\eps}(s,x)dxds
\\&=-k_2\int_0^t \int_\0 g''\big(I_{\lambda,\eps}(s,x)\big)\abs{\nabla I_{\lambda,\eps}(s,x)}^2dxds
\\&\leq 0.
\end{aligned}
\end{equation*}
Since $g(\xi)>0$ for all $\xi<0$, we conclude that $\forall \lambda \in \rho(A_1)\times \rho(A_2), \eps> 0$ then $ I_{\lambda,\eps}(t,x)\geq 0$ for all $t\in [0,T]$, almost everywhere in $\0$. Similarly, we have
\begin{equation*}
\begin{aligned}
\int_\0 g\big(S_{\lambda,\eps}(t,x)\big)dx&=\int_0^t\int_\0 g'\big(S_{\lambda,\eps}(s,x)\big)\Big(k_1\Delta S_{\lambda,\eps}(s,x)+\big(R_1(\lambda_1)\Lambda\big)(x)\Big)dxds
\\&=-k_1\int_0^t \int_\0 g''\big(S_{\lambda,\eps}(s,x)\big)\abs{\nabla S_{\lambda,\eps}(s,x)}^2dxds
\\&\quad+\int_0^t\int_\0 g'\big(S_{\lambda,\eps}(s,x)\big)\big(R_1(\lambda_1)\Lambda\big)(x)dxds
\\&\leq 0,
\end{aligned}
\end{equation*}
where the last inequality above follows from the fact $$R_1(\lambda_1,A_1)=\int_0^\infty
e^{-\lambda_1 t}e^{tA_1}dt$$ preserves positivity. Again, since $g(\xi)>0$ for all $\xi<0$, we obtain the positivity of $S_{\lambda,\eps}(t,x)$. Hence, $S^{*}(t,x),I^{*}(t,x)\geq 0$ almost everywhere in $\0$ for all $t\in [0,T]$, a.s.
\end{proof}

\para{Completion of the Proof of the Theorem.}
Since $\big(S^*(t),I^*(t)\big)$ is a unique mild solution of \eqref{S^*,I^*} and is positive,
it is a mild solution of \eqref{eq1}. Therefore, the equation \eqref{eq1} admits a unique positive mild solution $\big(S(t),I(t)\big)$.

Now, we prove the second part. For  convenience, we use subscripts to indicate the dependence of the solution on initial value. Let $Z_{z_0}(t),Z_{z_0'}(t)$ be the positive mild solutions of \eqref{so-vec} with the initial condition $Z(0)=z_0$ and $Z(0)=z_0'$, respectively. That means,
$$Z_{z_0}(t)=e^{tA}z_0+\int_0^t e^{(t-s)A}F^*\big(Z_{z_0}(s)\big)ds+\int_0^t e^{(t-s)A} Z_{z_0}(s)dW(s),$$
and
$$Z_{z_0'}(t)=e^{tA}z_0'+\int_0^t e^{(t-s)A}F^*\big(Z_{z_0'}(s)\big)ds+\int_0^t e^{(t-s)A}Z_{z_0'}(s)dW(s).$$
It implies that
\begin{equation*}
\begin{aligned}
Z_{z_0}(t)-Z_{z_0'}(t)&=e^{tA}z_0-e^{tA}z_0'+\int_0^t e^{(t-s)A}\Big(F^*\big(Z_{z_0}(s)\big)-F^*\big(Z_{z_0'}(s)\big)\Big)ds
\\&\quad+\int_0^t e^{(t-s)A}\Big(Z_{z_0}(s)-Z_{z_0'}(s)\Big)dW(s).
\end{aligned}
\end{equation*}
Since \eqref{lambda} and \eqref{boundedYbeta}, we can obtain that
\begin{equation}\label{3.10}
\begin{aligned}
\E \sup_{s\in[0,t]}&\abs{\int_0^s e^{(s-r)A}\Big(Z_{z_0}(r)-Z_{z_0'}(r)\Big)dW(r)}^p_{C(\bar\0,\R^2)}
\\&\leq c_p(t)\int_0^t \E\sup_{r\in [0,s]}\abs{Z_{z_0}(r)-Z_{z_0'}(r)}_{C(\bar\0,\R^2)}^pds
\\&\leq c_p(t)\int_0^t \abs{Z_{z_0}-Z_{z_0'}}_{L_{s,p}}^pds
\end{aligned}
\end{equation}
Therefore, by virtue of \eqref{estimateF} and \eqref{3.10}, it is possible to get
\begin{equation*}
\abs{Z_{z_0}-Z_{z_0'}}_{L_{t,p}}^p\leq c_p\abs{z_0-z_0'}_{C(\bar\0,\R^2)}^p+c_p(t)\int_0^t \abs{Z_{z_0}-Z_{z_0'}}_{L_{s,p}}^pds.
\end{equation*}
Hence, it is easy to obtain from Gronwall's inequality that
$$\abs{Z_{z_0}-Z_{z_0'}}_{L_{T,p}}^p\leq c_p(T)\abs{z_0-z_0'}_{C(\bar\0,\R^2)}^p.$$
Therefore, the continuous dependence of the solution on initial values is proved.
\end{proof}

\section{Longtime Behavior}\label{sec:longtime}
This section investigates the properties of the positive mild solution $\big(S(t),I(t)\big)$ of system \eqref{eq1} when $t\to\infty$. In particular, we provide the sufficient conditions for the extinction and permanence.
For each function $u\in E$,  denote $$u_*=\inf_{x\in\bar\0}u(x).$$
Define the number
$$\hat R=\int_\0\alpha(x)dx-\int_\0\mu_2(x)dx-\dfrac {a_2}2.$$

\begin{thm}\label{permanence}
If $\Lambda_*>0$ and $\hat R>0$, then the infected class is permanent in the sense that for any the initial values $0\leq S_0,I_0\in E$ satisfying
$$\int_\0 -\ln I_0(x)dx<\infty,$$
we have
$$
\liminf_{t\to\infty} \dfrac1t\int_0^t\left(\E\int_\0 \big(I^2(s,x)\wedge 1\big) dx\right)^{\frac12}ds\geq R_I,
$$
for some $R_I>0$ independent of initial values.
\end{thm}

\begin{proof}
To obtain the longtime properties of $\big(S(t),I(t)\big)$, one of tools we  use is It\^o's formula. Unfortunately, in general the It\^o's formula is not valid for the mild solutions. Hence, our idea is to approximate the solution by a sequence of strong solutions when the noise is finite dimensional. First, we assume that $S_0,I_0\in D(A_i^E)$, where $D(A_i^E)$ is the domain of $A_i^E$, the part of $A_i$ in $E$.
For each fixed $n\in \N$, let $\bar S_n(t,x),\bar I_{n}(t,x)$ be the strong solution (see
\cite{Prato}
for more details about strong solutions, weak solutions, and mild solutions) of the following equations
\begin{equation}\label{StrongS_n}
\begin{cases}
d\bar S_n(t,x)=\Big[A_1\bar S_n(t,x)+\Lambda(x)-\mu_1(x)\bar S_n(t,x)-\dfrac{\alpha(x) \bar S_n(t,x)\bar I_n(t,x)}{ \bar S_n(t,x)+\bar I_n(t,x)}\Big]dt \\
\hspace{7.5cm}+ \displaystyle \sum_{k=1}^n \sqrt{a_{k,1}}e_k(x)\bar S_n(t,x)dB_{k,1}(t),\\
d\bar I_n(t,x)=\Big[A_2 \bar I_n(t,x)-\mu_2(x) \bar I_n(t,x) + \dfrac{\alpha(x) \bar S_n(t,x)\bar I_n(t,x)}{ \bar S_n(t,x)+\bar I_n(t,x)}\Big]dt\\
\hspace{7.5cm}+\displaystyle \sum_{k=1}^n \sqrt{a_{k,2}}e_k(x)\bar I_n(t,x)dB_{k,2}(t),\\
\bar S_n(x,0)=S_0(x),\quad \bar I_n(x,0)=I_0(x).
\end{cases}
\end{equation}
The existence and uniqueness of the strong solution of \eqref{StrongS_n} follow the results in \cite{Prato85} or \cite[Section 7.4]{Prato}. To see that the conditions in these references are satisfied, we note that the semi-groups $e^{tA_1},e^{tA_2}$ (as well as their restrictions to $E$) are analytic (see \cite[Chapter 2]{Arendt}) and strongly continuous (see \cite[Appendix A.5.2]{Prato}). Moreover, since the characterizations of fractional power of elliptic operators in (\cite[Chapter 16]{Yagi} or \cite[Appendix A]{Prato}, it is easy to confirm that the coefficients in equation \eqref{StrongS_n} satisfies condition (e) in Hypothesis 2 in \cite{Prato85}.
Moreover, a detailed argument can be also found in \cite{NY19,NY20}.

In addition, since the continuous dependence on parameter $\xi$ of the fixed points of family of uniform contraction mappings $T(\xi)$, by a similar ``parameter-dependent contraction mapping" argument, it is easy to obtain that (see \cite{Prato} or \cite[Proposition 4.2]{NY20}) for any fixed $t$,
$$\lim_{n\to\infty}\E\abs{S(t)-\bar S_n(t)}_H^2\to 0,$$
and
\begin{equation*}
\lim_{n\to\infty}\E\abs{I(t)-\bar I_n(t)}_H^2\to 0.
\end{equation*}

To proceed, we state and prove following auxiliary Lemmas.

\begin{lem}
Let
$$\mu_*:=\inf_{x\in\bar\0}\min\{\mu_1(x),\mu_2(x)\}.$$
If $\mu_*>0$ then
$$
\E\int_\0\big(S(t,x)+ I(t,x)\big)dx\leq e^{-\mu_* t} \int_\0\big(S_0(x)+I_0(x)\big)dx+ \dfrac{ |\Lambda|_E}{\mu_*}.
$$
\end{lem}

\begin{proof}
In view of It\^o's formula (\cite[Theorem 3.8]{Curtain}), we can obtain
$$
\begin{aligned}
\E e^{\mu_*t}\int_\0\big(\bar S_n(t,x)+ \bar I_n(t,x)\big)dx
&\leq \int_\0\big(S_0(x)+I_0(x)\big)dx+\E\int_0^t e^{\mu_*s}\int_\0\Lambda(x)dx ds\\&
\leq \int_\0\big(S_0(x)+I_0(x)\big)dx+\dfrac{ |\Lambda|_E}{\mu_*}e^{\mu_* t}.
\end{aligned}
$$
Letting $n\to\infty$, we obtain the desired result.
\end{proof}

Now, we  are in a position to estimate $\E \int_\0 \frac 1{\bar S_n^p(t ,x)}dx$ by the following Lemma.

\begin{lem}\label{lm4.1}
For any $p>0$, if $\int_\0\frac 1{S_0^p(x)}dx<\infty$, there exists $\wdt K_p>0$, which is independent of $n$ and initial conditions such that
\begin{equation*}
\begin{aligned}
\E&\int_\0 \dfrac 1{\bar S_n^p(t ,x)}dx
\leq e^{-t}\int_\0\frac 1{S_0^p(x)}dx+\wdt K_p.
\end{aligned}
\end{equation*}
\end{lem}

\begin{proof}
For any $0<\eps<\frac{p\Lambda_*}{2}$, using It\^o's Lemma (\cite[Theorem 3.8]{Curtain}) and by direct calculations, we have
\begin{equation}\label{1/s^2}
\begin{aligned}
&\!\!\! e^{t }\int_\0 \dfrac 1{\big(\bar S_n(t ,x)+\eps\big)^p}dx \\ & \ =\int_\0\dfrac 1{\big(S_0(x)+\eps\big)^p}dx+\int_0^{t } e^s\int_\0 \dfrac 1{\big(\bar S_n(s,x)+\eps\big)^p}dxds +\int_0^{t } e^s\int_\0\dfrac {-p}{\big(\bar S_n(s,x)+\eps\big)^{p+1}}\\
\\&\qquad \times \Bigg(k_1\Delta \bar S_n(s,x)+\Lambda(x)-\mu_1(x)\bar S_n(s,x)-\dfrac{\alpha(x)\bar S_n(s,x)\bar I_n(s,x)}{\bar S_n(s,x)+\bar I_n(s,x)}\Bigg)dxds
\\&\qquad+\dfrac 12 \int_0^{t } e^s\sum_{k=1}^n\int_\0 \dfrac{p(p+1)a_{k,1}e^2_k(x)\bar S_n^2(s,x)}{\big(\bar S_n(s,x)+\eps\big)^{p+2}}dxds\\
&\qquad +\sum_{k=1}^n\int_0^{t } e^s\Big[\sqrt{a_{k,1}}\int_\0 \dfrac{-pe_k(x)\bar S_n(s,x)}{\big(\bar S_n(s,x)+\eps\big)^{p+1}}dx\Big]dB_{k,1}(s)
\\&\leq \int_\0\dfrac 1{\big(S_0(x)+\eps\big)^p}dx+\int_0^{t } e^s\int_\0\dfrac{-pk_1\Delta\bar S_n(s,x)}{\big(\bar S_n(s,x)+\eps\big)^{p+1}}dxds
\\&\qquad+\int_0^{t } e^s\int_\0 \dfrac{p}{\big(\bar S_n(s,x)+\eps\big)^{p+1}}\Big(-\Lambda (x)+\frac{\eps}p+\big(\abs{\mu_1}_E+\abs{\alpha}_E+\dfrac 1p+\dfrac{p+1}{2}a_1C_0^2\big)
\\&\qquad \times \bar S_n(s,x)\Big)dxds+\sum_{k=1}^n\int_0^{t } e^s\Big[\sqrt{a_{k,1}}\int_\0 \dfrac{-pe_k(x)\bar S_n(s,x)}{\big(\bar S_n(s,x)+\eps\big)^{p+1}}dx\Big]dB_{k,1}(s)
\\&\leq \int_\0\dfrac 1{\big(S_0(x)+\eps\big)^p}dx+\int_0^{t } \dfrac{pK_p^{p+1}2^p}{\Lambda_*^p}e^sds\\
&\qquad +\sum_{k=1}^n\int_0^{t } e^s\Big[\sqrt{a_{k,1}}\int_\0 \dfrac{-pe_k(x)\bar S_n(s,x)}{\big(\bar S_n(s,x)+\eps\big)^{p+1}}dx\Big]dB_{k,1}(s),
\end{aligned}
\end{equation}
where $K_p=\abs{\mu_1}_E+\abs{\alpha}_E+\frac 1p+\frac{p+1}{2}a_1C_0^2$.
In the above, we used the following facts
$$\int_\0\dfrac{-pk_1\Delta\bar S_n(s,x)}{\big(\bar S_n(s,x)+\eps\big)^{p+1}}dx=-p(p+1)k_1\int_\0\dfrac{\abs{\nabla \bar S_n(s,x)}^2}{\big(\bar S_n(s,x)+\eps\big)^{p+2}}dx\leq 0 \ \hbox{ a.s.,}$$
and
\begin{equation*}
\begin{aligned}
\int_\0 &\dfrac{p}{(\bar S_n(s,x)+\eps)^{p+1}}\Big(-\Lambda (x)+\dfrac{\eps}p+\big(\abs{\mu_1}_E+\abs{\alpha}_E+\dfrac 1p+\dfrac{p+1}{2}a_1C_0^2\big)\bar S_n(s,x)\Big)dx
\\&\leq\int_\0\dfrac{p}{\big(\bar S_n(s,x)+\eps\big)^{p+1}}\Big(-\dfrac{\Lambda_*}2+K_p\bar S_n(s,x)\Big)\1_{\{\bar S_n(s,x)\geq \frac{\Lambda_*}{2K_p}\}}dx
\\&\leq \dfrac{pK_p^{p+1}2^p}{\Lambda_*^p}\a.s
\end{aligned}
\end{equation*}
Hence, \eqref{1/s^2} implies that $\forall t\geq 0$, $\forall n\in \N$
\begin{equation}\label{1/s+eps}
\begin{aligned}
\E \int_\0 \dfrac 1{\big(\bar S_n(t ,x)+\eps\big)^p}dx\leq e^{-t}\int_\0\dfrac 1{\big(S_0(x)+\eps\big)^p}dx+ e^{-t }\int_0^{t }\dfrac{pK_p^{p+1}2^p}{\Lambda_*^p}e^sds.
\end{aligned}
\end{equation}
Letting $\eps\to 0$, we have from the monotone convergence theorem that
\begin{equation}\label{e1/s^2}
\begin{aligned}
\E \int_\0 \dfrac 1{\bar S_n^p(t ,x)}dx\leq e^{-t}\int_\0\frac 1{S_0^p(x)}dx+ e^{-t }\int_0^{t }\dfrac{pK_p^{p+1}2^p}{\Lambda_*^p}e^sds.
\end{aligned}
\end{equation}
The proof of the Lemma is completed.
\end{proof}

Noting that our initial condition are not assumed to satisfy $\int_\0 \frac 1{S^2_0(x)}dx<\infty$. However, we will prove that after some finite time, the solutions have the inverse functions that belong to $L^2(\0,\R)$ as the following Lemma.

\begin{lem}\label{lm4.2}
For any $n\in\N$
\begin{equation*}
\begin{aligned}
\E &\int_\0 \dfrac 1{\bar S_n^2(4 ,x)}dx
\leq \ell_0,
\end{aligned}
\end{equation*}
 where $\ell_0$ depends only initial condition (independent of $n$).
\end{lem}

\begin{proof}
By the following facts:
\begin{equation*}
\begin{aligned}
\E\int_\0 \bar S_n(t,x)dx&=\int_\0 S_0(x)dx +\int_0^t \E\int_\0 \Big(k_1\Delta \bar S_n(s,x)+\Lambda(x)
\\&\quad\quad\quad-\mu_1(x)\bar S_n(s,x)-\dfrac{\alpha(x)\bar S_n(s,x)\bar I_n(s,x)}{\bar S_n(s,x)\bar I_n(s,x)}\Big)dxds
\\&\leq \int_\0 S_0(x)dx +t\abs{\Lambda}_E,
\end{aligned}
\end{equation*}
and  $s^q\leq s+1,\;\forall s\in\R>0, q \in [0,1]$, it is easy to show that there exists $\ell_{1}>0$ such that
\begin{equation}\label{e1-lm4.2}
\E \int_\0 \bar S_n^q(t ,x)dx
\leq \ell_{1}, \text{ for any } t\in[0,1], q\in[0,1],
\end{equation}
where $\ell_1$ is independent of $n$. For any $\eps>0$, using It\^o's Lemma (\cite[Theorem 3.8]{Curtain}) again, we have
\begin{equation}\label{e2-lm4.2}
\begin{aligned}
&\!\!\!\E \int_\0 \big(\bar S_n(1 ,x)+\eps\big)^{\frac12}dx \\ & \ =\int_\0\big(S_0(x)+\eps\big)^{\frac12}dx+\int_0^{1}\E \int_\0\dfrac {1}{2\big(\bar S_n(s,x)+\eps\big)^{\frac12}}\Big(k_1\Delta \bar S_n(s,x)\\
\\&\qquad+\Lambda(x)-\mu_1(x)\bar S_n(s,x)-\dfrac{\alpha(x)\bar S_n(s,x)\bar I_n(s,x)}{\bar S_n(s,x)+\bar I_n(s,x)}\Big)dxds
\\&\qquad-\dfrac 18 \int_0^{1} \E\sum_{k=1}^n\int_\0 \dfrac{a_{k,1}e^2_k(x)\bar S_n^2(s,x)}{\big(\bar S_n(s,x)+\eps\big)^{\frac32}}dxds.\\
&\geq \frac12\int_0^1\E \int_\0 \dfrac{\Lambda(x)}{\big(\bar S_n(s,x)+\eps\big)^{\frac12}}dxds
-N_1  \int_0^1\left(\E\int_\0 \bar S_n^{\frac12}(s ,x)dx\right)ds,
\end{aligned}
\end{equation}
where $$N_1=\dfrac{\abs{\mu_1}_E+\abs{\alpha}_E+\frac {a_1C_0^2}4}{2}.$$
In view of \eqref{e1-lm4.2} and \eqref{e2-lm4.2}, we have
$$
\frac12\int_0^1\E \int_\0 \dfrac{\Lambda(x)}{\big(\bar S_n(s,x)+\eps\big)^{\frac12}}dxds\leq (1+N_1)\ell_{1}+\sqrt{\eps},\quad\forall \eps>0,$$
which implies that
$$
\int_0^1\E \int_\0 \dfrac{\Lambda(x)}{\big(\bar S_n(s ,x)\big)^{\frac12}}dxds\leq 2(1+N_1)\ell_{1}.$$
or there exists $t_1=t_1(n)\in[0,1]$
such that
$$
\E \int_\0 \dfrac{\Lambda(x)}{\big(\bar S_n(t_1,x)\big)^{\frac12}}dx\leq \dfrac{2(1+N_1)\ell_{1}}{\Lambda_*}.
$$
Applying Lemma \ref{lm4.1} and the Markov property of $(S_n, I_n)$, we have
$$
\E \int_\0 \dfrac{\Lambda(x)}{\big(\bar S_n(t ,x)+\eps\big)^{\frac12}}dx\leq\ell_{2}, \;\forall t\in[1,2],
$$
for some $\ell_{2}$ independent of $n$. We again have
\begin{equation}\label{e3-lm4.2}
\begin{aligned}
&\!\!\!\E \int_\0 \big(\bar S_n(2 ,x)+\eps\big)^{-\frac12}dx \\ & \ =\E\int_\0(\bar S_n(1,x)+\eps)^{-\frac12}dx -\int_1^{2}\E \int_\0\dfrac {1}{2\big(\bar S_n(s,x)+\eps\big)^{\frac32}}\Big(k_1\Delta \bar S_n(s,x)\\
\\&\qquad+\Lambda(x)-\mu_1(x)\bar S_n(s,x)-\dfrac{\alpha(x)\bar S_n(s,x)\bar I_n(s,x)}{\bar S_n(s,x)+\bar I_n(s,x)}\Big)dxds
\\&\qquad+\dfrac 38 \int_1^{2 } \E\sum_{k=1}^n\int_\0 \dfrac{a_{k,1}e^2_k(x)\bar S_n^2(s,x)}{\big(\bar S_n(s,x)+\eps\big)^{\frac32}}dxds\\
&\leq -\frac12\int_1^2\E \int_\0 \dfrac{\Lambda(x)}{\big(\bar S_n(s ,x)+\eps\big)^{\frac32}}dxds
+\E\int_\0\big(\bar S_n(1,x)+\eps\big)^{-\frac12}dx\\
&\qquad +N_2  \int_1^2\left(\E\int_\0 \bar S_n^{-\frac12}(s ,x)dx\right)ds,
\end{aligned}
\end{equation}
where $$N_2=\dfrac{\abs{\mu_1}_E+\abs{\alpha}_E+\frac {3a_1C_0^2}4}{2}.$$
 Thus,
$$\int_1^2\E \int_\0 \dfrac{\Lambda(x)}{\big(\bar S_n(s ,x)+\eps\big)^{\frac32}}dxds\leq \ell_{3},
$$
for some $\ell_{3}$ depending only on initial conditions. Letting $\eps\to0$, we can obtain that for some $t_2=t_2(n)\in[1,2]$
$$\E \int_\0 \dfrac{\Lambda(x)}{\bar S_n^{\frac32}(t_2,x)}dx\leq \ell_{3},$$
which together with Lemma \ref{lm4.1} implies that
$$\E \int_\0 \dfrac{\Lambda(x)}{\big(\bar S_n(t ,x)+\eps\big)^{\frac32}}dx\leq \ell_{4},\; \forall t\in[2,3],$$
where $\ell_4$ is some constant independent of $n$. Keeping this process we can obtain that there exists $t_3=t_3(n)\in [0,4]$, $\ell_5$ such that
$$\E \int_\0 \dfrac{\Lambda(x)}{(\bar S_n(t_3 ,x))^{\frac52}}dx\leq \ell_{5}.$$
Therefore, it is possible to obtain the existence of two constants $t_4=t_4(n)\in [0,4]$ and $\ell_6$ satisfying
$$\E \int_\0 \dfrac 1{\big(\bar S_n(t_4,x)\big)^2}dx<\ell_6.$$
The Lemma is proved by applying Lemma \ref{lm4.1}.
\end{proof}

In view of Lemma \ref{lm4.1} and Lemma \ref{lm4.2}, we have
\begin{equation}\label{1/sn^2}
\begin{aligned}
\E \int_\0 \dfrac 1{\bar S_n^2(t ,x)}dx\leq e^{-t}\ell_0+ \wdt K_2\;\forall n\in\N, t\geq 4.
\end{aligned}
\end{equation}
Noting that both $\ell_0$ and $\wdt K_2$ are independent of $n$; and $\ell_0$ may depend on initial point but $\wdt K_2$ is independent.
By It\^o's Lemma (\cite[Theorem 3.8]{Curtain}) again and similar calculations in the process of getting \eqref{1/s+eps} we have
\begin{equation*}
\begin{aligned}
\E \int_\0\bar I_n(t,x)dx\geq&
\E \int_\0\ln \big(\bar I_n(t,x)+\eps\big)dx\\
&=\int_\0\ln \big(I_0(x)+\eps\big)dx+\int_0^t \E\int_\0\dfrac 1{\bar I_n(s,x)+\eps}\Big(k_2\Delta \bar I_n(s,x)
\\&\quad\quad\quad-\mu_2(x)\bar I_n(s,x)+\dfrac{\alpha(x)\bar S_n(s,x)\bar I_n(s,x)}{\bar S_n(s,x)+\bar I_n(s,x)}dxds\Big)
\\&\quad\quad\quad-\dfrac 12\int_0^t\E\sum_{k=1}^n \int_\0\dfrac{a_{k,2}\bar I_n^2(s,x)e_k^2(x)}{\big(\bar I_n(s,x)+\eps\big)^2}dxds
\\&\geq \int_\0\ln \big(I_0(x)+\eps\big)dx-\big(\dfrac {a_2}2+\abs{\mu_2}_E\big)t,\;\forall n\in\N, \forall t>0, 0<\eps<1.
\end{aligned}
\end{equation*}
As a consequence
\begin{equation}\label{lb-lnI}
\E \int_\0\bar I_n(t,x)dx\geq\E\int_\0 \ln \bar I_n(t,x)dx\geq\int_\0\ln I_0(x)dx-\big(\dfrac{a_2}2+\abs{\mu_2}_E\big)t>-\infty,\;\forall t>0.
\end{equation}
That means
\begin{equation}\label{I>0a.e}
\PP\big\{\bar I_n(t,x)>0\;\text{almost everywhere in}\;\0\big\}=1,\;\forall n\in\N, \forall t>0.
\end{equation}
On the other hand, combining It\^o's Lemma and basic calculations implies that
\begin{equation*}
\begin{aligned}
0&\geq \E\int_\0\ln \dfrac{\bar I_n(t,x)+\eps}{1+\bar I_n(t,x)}dx\geq \int_\0 \ln\dfrac{I_0(x)+\eps}{1+I_0(x)}dx+\hat Rt
\\&\quad\quad-\int_0^{t}\E\int_\0 \Big(\dfrac{\alpha(x) \bar I_n(s,x)}{\bar S_n(s,x)+\bar I_n(s,x)}+\dfrac{\alpha(x) \bar S_n(s,x)\bar I_n(s,x)}{\big(\bar S_n(s,x)+\bar I_n(s,x)\big)\big(\bar I_n(s,x)+1\big)}\Big)dxds
\\&\quad\quad-\int_0^t\E\int_\0 \dfrac{\alpha(x)\eps}{\bar I_n(s,x)+\eps}dxds,\;\forall t>0,n\in \N, 0<\eps<1.
\end{aligned}
\end{equation*}
Thus, $\forall t> 0, n\in\N$, $0<\eps<1$
\begin{equation}\label{p-e0}
\begin{aligned}
\int_0^t \E\int_\0 &\Big(\dfrac{\alpha(x)\bar I_n(s,x)}{\bar S_n(s,x)+\bar I_n(s,x)}+\dfrac{\alpha(x)\bar S_n(s,x)\bar I_n(s,x)}{\big(\bar S_n(s,x)+\bar I_n(s,x)\big)\big(\bar I_n(s,x)+1\big)}\Big)dxds\\
\geq& \E\int_\0 \ln\dfrac{I_0(x)+\eps}{1+I_0(x)}dx+\hat Rt-\abs{\alpha}_E\int_0^t \E\int_\0 \dfrac{\eps}{\bar I_n(s,x)+\eps}dxds.
\end{aligned}
\end{equation}
Let $\eps\to 0$ and using \eqref{I>0a.e} and \eqref{p-e0} we have
\begin{equation}\label{p-e1}
\begin{aligned}
\int_0^t \E\int_\0 &\Big(\dfrac{\alpha(x)\bar I_n(s,x)}{\bar S_n(s,x)+\bar I_n(s,x)}+\dfrac{\alpha(x)\bar S_n(s,x)\bar I_n(s,x)}{\big(\bar S_n(s,x)+\bar I_n(s,x)\big)\big(\bar I_n(s,x)+1\big)}\Big)dxds\\
\geq& \int_\0\ln \dfrac{I_0(x)}{1+I_0(x)}dx+\hat Rt,\;\forall t>0,n\in\N.
\end{aligned}
\end{equation}
We have the following estimates:
\begin{equation*}
\begin{aligned}
\abs{\alpha}_E\left(\E\int_\0 \dfrac{\bar I_n^2(s,x)}{\big(1+\bar I_n(s,x)\big)^2}dx\right)^{\frac12}
&\geq \E\int_\0\dfrac{\alpha(x)\bar I_n(s,x)}{1+\bar I_n(s,x)}dx\\
&\geq\E \int_\0 \dfrac{\alpha(x)\bar S_n(s,x)\bar I_n(s,x)}{\big(\bar S_n(s,x)+\bar I_n(s,x)\big)\big(\bar I_n(s,x)+1\big)}dx,
\end{aligned}
\end{equation*}
and
$$
\begin{aligned}
\abs{\alpha}_E&\left(\E\int_\0 \dfrac{\bar I_n^2(s,x)}{\big(1+\bar I_n(s,x)\big)^2}dx\right)^{\frac12}\left(\E \int_\0\left(\dfrac 1{\bar S_n(s,x)}+1\right)^2dx\right)^{\frac12}\\
&\geq
 \E\int_\0\dfrac{\alpha(x)\bar I_n(s,x)}{1+\bar I_n(s,x)}\left(\dfrac1{\bar S_n(s,x)}+1\right)dx
 \geq
 \E\int_\0\dfrac{\alpha(x)\bar I_n(s,x)}{\bar S_n(s,x)+\bar I_n(s,x)}dx,
\end{aligned}
$$
since $$\frac{1+I}{S+I}=\frac{1}{S+I}+\frac{I}{S+I}\leq \frac1S+1.$$
Therefore, after some basic estimates, we can get from \eqref{p-e1} that
\begin{equation*}
\begin{aligned}
\int_4^t \abs{\alpha}_E&\left(\E\int_0 \dfrac{\bar I_n^2(s,x)}{\big(1+\bar I_n(s,x)\big)^2}dx\right)^{\frac12}\left(1+\left(\E \int_\0\left(\dfrac 1{\bar S_n(s,x)}+1\right)^2dx\right)^{\frac12}\right)ds\\
\geq&
\int_4^t \E\int_\0 \Big(\dfrac{\alpha(x)\bar I_n(s,x)}{\bar S_n(s,x)+\bar I_n(s,x)}+\dfrac{\alpha(x)\bar S_n(s,x)\bar I_n(s,x)}{\big(\bar S_n(s,x)+\bar I_n(s,x)\big)\big(\bar I_n(s,x)+1\big)}\Big)dxds\\
\geq&
\int_\0\ln \dfrac{I_0(x)}{1+I_0(x)}dx+\hat R t- 8\abs{\alpha}_E,
\end{aligned}
\end{equation*}
which together with \eqref{1/sn^2} leads to
\begin{equation*}
\begin{aligned}
\int_4^t \abs{\alpha}_E&\left(\E\int_0 \dfrac{\bar I_n^2(s,x)dx}{\big(1+\bar I_n(s,x)\big)^2}\right)^{\frac12}\left(2\sqrt{e^{-s}\ell_0}+2\wdt K_2^{\frac 12}+3\right) ds\\
\geq&
 \int_\0\ln \dfrac{I_0(x)}{1+I_0(x)}dx-8\abs{\alpha}_E+\hat R t.
\end{aligned}
\end{equation*}
Letting $n\to\infty$ yields
\begin{equation}\label{i/1+i}
\begin{aligned}
\int_4^t \abs{\alpha}_E&\left(\E\int_\0 \dfrac{I^2(s,x)}{\big(1+I(s,x)\big)^2}dx\right)^{\frac12}\left(2\sqrt{e^{-s}\ell_0}+2\wdt K_2^{\frac 12}+3\right) ds\\
\geq& \int_\0\ln \dfrac{I_0(x)}{1+I_0(x)}dx-8\abs{\alpha}_E+\hat R t,
\end{aligned}
\end{equation}
which is easily followed by
\begin{equation*}
\begin{aligned}
\liminf_{t\to\infty}\dfrac1t\int_0^t \left(\E\int_\0 \dfrac{I^2(s,x)}{\big(1+I(s,x)\big)^2}dx\right)^{\frac12}ds
\geq \dfrac{\hat R}{\abs{\alpha}_E\big(2\wdt K_2^{\frac 12}+3\big)}.
\end{aligned}
\end{equation*}
As a consequence,
\begin{equation*}
\liminf_{t\to\infty}\dfrac1t\int_0^t\left(\E\int_\0 \big(I^2(s,x)\wedge 1\big) dx\right)^{\frac12}ds\geq R_I>0,
\end{equation*}
where $R_I$ is independent of initial points. The proof of the theorem is completed by using dense property of $D(A_i^E)$ in $E$ and continuous dependence on initial data of the solution. In more detailed, since constants $\wdt K_2,\hat R$ are independents of initial points, the estimates \eqref{1/sn^2} and \eqref{i/1+i} still hold for the solution starting from arbitrary initial points $S_0,I_0\in E$ with $\int_\0 -\ln I_0(x)dx<\infty.$
\end{proof}

\begin{thm}\label{extinction}
For any nonnegative initial data $S_0,I_0\in E$, if \begin{equation}\label{eq:min}(\mu_2-\alpha)_*= \inf_{x\in\bar\0}\big(\mu_2(x)-\alpha(x)\big)>0,\end{equation}
then the infected class will be extinct with exponential rate.
\end{thm}

\begin{proof}
 First, we define the linear operator $J:H \mapsto
\R$ as following
$$\forall u\in H, Ju:=\int_\0 u(x)dx.$$
By the properties of $e^{tA_i}$,
 $\forall u\in H, J(e^{tA_i}u-u)=0$ or $Ju=Je^{tA_i}u, \forall i=1,2.$

Now, as in the definition of mild solution, we have
$$I(t)=e^{tA_2}I_0+\int_0^te^{(t-s)A_2}\Big(-\mu_2I(s)+\dfrac{\alpha S(s)I(s)}{S(s)+I(s)}\Big)ds+\int_0^t e^{(t-s)A_2}I(s)dW_2(s).$$
Hence, applying
the operator $J$ to both sides, using the properties of operator $J$ and stochastic convolution (see \cite[Proposition 4.15]{Prato}), we obtain
\begin{equation*}
\begin{aligned}
\int_\0I(t,x)dx=&\int_\0I_0(x)dx+\int_0^t \int_\0\Big(-\mu_2(x) I(s,x)+\dfrac{\alpha(x) S(s,x)I(s,x)}{S(s,x)+I(s,x)}\Big)dxds
\\&\quad+\int_0^t J(e^{(t-s)A_2}I(s))dW_2(s),
\end{aligned}
\end{equation*}
where $J(e^{(t-s)A_2}I(s))$ in the stochastic integral is understood as the process taking values in spaces of linear operator from $H$ to $\R$, that is defined by $$J(e^{(t-s)A_2}I(s))u:=\int_\0 \Big(e^{(t-s)A_2}I(s)u\Big)(x)dx\quad\forall u\in H.$$
Since \eqref{condition}, it is easy to see that these integrals are well-defined.
By taking the expectation on both sides and using the properties of stochastic integral \cite[Proposition 2.9]{Curtain},
\begin{equation*}
\begin{aligned}
\E \int_\0I(t,x)dx&=\int_\0I_0(x)dx+\E\int_0^t\int_\0\Big(-\mu_2(x) I(s,x)+\dfrac{\alpha(x) S(s,x)I(s,x)}{S(s,x)+I(s,x)}\Big)dxds
\end{aligned}
\end{equation*}
As a consequence,
\begin{equation}
\begin{aligned}
\E \int_\0I(t,x)dx-\E \int_\0 I(s,x)dx&=\int_s^t\E\int_\0\Big(-\mu_2(x) I(r,x)+\dfrac{\alpha(x) S(r,x)I(s,x)}{S(r,x)+I(r,x)}\Big)dxdr
\\&\leq -(\mu_2-\alpha)_*\int_s^t\E\int_\0 I(r,x)dxdr
\end{aligned}
\end{equation}
Hence, we can obtain the following estimate for the upper Dini derivative
\begin{equation*}
\dfrac{d}{dt^+}\E\int_\0I(t,x)dx\leq -(\mu_2-\alpha)_* \E\int_\0I(t,x)dx,\;\forall t\geq 0.
\end{equation*}
Since $(\mu_2-\alpha)_*>0$, we can get that $\E\int_\0I(t,x)dx$ converges to 0 with exponential rate as $t\to \infty$. Hence, it easy to claim that the infected class goes extinct.
\end{proof}

\begin{thm}\label{extinction2}
Suppose that $W_2(t)$ is a space-independent Brownian motion
with covariance $a_2 t$.
For any nonnegative initial data $S_0,I_0\in E$, if \begin{equation}\label{eq:min2}(\mu_2-\alpha)_*
+\frac{a_2}2:=\inf_{x\in\bar\0}\big(\mu_2(x)-\alpha(x)\big)
+\frac{a_2}2>0,\end{equation}
then when $p>0$ be sufficiently small that
$$R_p:=(\mu_2-\alpha)_*+\frac{(1-p)a_2}2>0,$$
we have
$$\limsup_{t\to\infty}\dfrac{\ln \E\left(\int_\0 I(t,x)dx\right)^p}t\leq -pR_p<0.$$

\end{thm}

\begin{proof}
Since $W_2(t)$ is a space-independent Brownian motion, as the arguments in proof of Theorem \ref{permanence}, the mild solution $I(t)$ is also the solution in the strong sense if $I_0\in D(A_i^E)$. Hence, with initial value in $D(A_i^E)$, we have
$$
\int_\0 I(t,x)=\int_0^t\int_\0\Big(-\mu_2(x) I(s,x)+\dfrac{\alpha(x) S(s,x)I(s,x)}{S(s,x)+I(s,x)}\Big)dxds +\int_0^t\int_\0 I(s,x)dW_2(s)
$$
By It\^o's formula, we obtain that
$$
\begin{aligned}
\Big(&\int_\0 I(t,x)dx\Big)^p\\
&= \int_s^t\left( p\Big(\int_\0 I(r,x)dx\Big)^{p-1}\int_\0\Big(-\mu_2(x) I(r,x)+\dfrac{\alpha(x) S(r,x)I(r,x)}{S(r,x)+I(r,x)}\Big)dx\right)dr\\
& +\int_s^t p(1-p)\frac{a_2}2  \left(\int_\0I(r,x)dx\right)^pdr + \int_s^t\left(\int_\0I(r,x)dx\right)^pdW_2(r)\\
&\leq -pR_p \int_s^t \left(\int_\0I(r,x)dx\right)^pdr + \int_s^t\left(\int_\0I(r,x)dx\right)^pdW_2(r).
\end{aligned}
$$
Since $\E \left(\int_\0I(t,x)dx\right)^p<\infty$,
we have
$$
\E \left(\int_\0 I(t,x)dx\right)^p= \E\left(\int_\0 I(s,x)dx\right)^p-pR_p \int_s^t \E\left(\int_\0I(r,x)dx\right)^pdr
$$
which easily derives that
$$\dfrac{d}{dt^+} \E \left(\int_\0 I(t,x)dx\right)^p\leq -pR_p  \E \left(\int_\0 I(t,x)dx\right)^p.$$
An application of the differential inequality shows
\begin{equation}\label{exp-p}
\E \left(\int_\0 I(t,x)dx\right)^p \leq e^{-pR_pt} \left(\int_\0 I(0,x)dx\right)^p,
\end{equation}
for any $t\geq 0$ and initial values in $D(A_i^E)$.
Since $D(A_i^E)$ is dense in $E$, \eqref{exp-p} holds for each fixed $t$ and any initial values in $E$.
Then the desired result can be obtained.
\end{proof}

\section{An Example}
\label{example}
In this section, to demonstrate our results, we consider an example when the processes driving noise processes in
equation \eqref{eq} are standard Brownian motions
and the recruitment rate, the
death rates, the infection rate, and the recovery rate are independent of space variable. Precisely, we consider the following equation
\begin{equation}\label{eq-1}
\begin{cases}
dS(t,x)=\Big[k_1\Delta S(t,x)+\Lambda-\mu_1S(t,x)- \dfrac{\alpha S(t,x)I(t,x)}{S(t,x)+ I(t,x)}\Big]dt \\
\hspace{7.5cm}+ \sigma_1S(t,x)dB_1(t)\quad
 \text{in }\ \R^+\times\0, \\[1ex]
dI(t,x)=\Big[k_2\Delta I(t,x)-\mu_2I(t,x) +\dfrac{\alpha S(t,x)I(t,x)}{S(t,x)+ I(t,x)}\Big]dt\\
\hspace{7.5cm}+\sigma_2 I(t,x)dB_2(t)\quad\text{in } \ \R^+\times\0,\\[1ex]
\partial_{\nu}S(t,x)=\partial_{\nu}I(t,x)=0\quad\quad\quad\quad\quad\quad\quad\text{in} \;\;\;\;\R^+\times\partial\0,\\
S(x,0)=S_0(x),I(x,0)=I_0(x)\quad\quad\quad\quad\;\text{in} \;\;\;\;\bar\0,
\end{cases}
\end{equation}
where $\Lambda,\mu_1,\mu_2,\alpha$ are positive constants, and $B_1(t)$, $B_2(t)$ are independent standard Brownian motions.
As we obtained above, for any initial values $0\leq S_0,I_0\in E$,
\eqref{eq-1} has unique positive mild solution $S(t,x),I(t,x)\geq 0.$ Moreover, the long-time behavior of the system is shown as the following theorem.
\begin{thm}\label{longtimeexample}
Let $S(t,x),I(t,x)$ be the positive mild solution $($in fact also in the strong sense$)$ of equation \eqref{eq-1}.
\begin{itemize}
\item[{\rm(i)}] For any non-negative initial values $S_0,I_0\in E$, if $\alpha<\mu_2+\dfrac {\sigma_2^2}2$, then the infected individual is extinct.
\item[{\rm(ii)}] For the initial values $0\leq S_0,I_0\in E$ satisfy $$\int_\0-\ln I_0(x)dx<\infty.$$ If  $\alpha>\mu_2+\dfrac {\sigma_2^2}2$, then the infected class is permanent.
\end{itemize}
\end{thm}

\begin{rem}
As in  Theorem \ref{longtimeexample}, the sufficient condition for permanence is almost necessary condition. It is similar to the result for SIS reaction-epidemic model, which is shown in \cite[Theorem 1.2]{Peng}.
\end{rem}

\section{Concluding Remarks}\label{sec:con}
 Being possibly among one of the first papers working on spatially inhomogeneous stochastic partial differential equation epidemic models,
we hope that our effort will provide some insights for subsequent study and investigation.
For possible future study, we mention the following topics.
\begin{itemize}
 \item First,  there is a growing interest to use the so-called regime-switching stochastic models in various applications; see \cite{YinZ} for the treatment of switching diffusion models, in which both continuous dynamics and discrete events coexist. Such switching diffusion models have gained popularity with applications range from networked control systems to financial engineering.  For instance, in a financial market model, one may use the random switching process to model the mode of the market (bull and bear). Such a random switching process can be built into the SPDE models considered here.
     The switching is used to reflect different random environment that are not reflected from the SPDE part of the model.

 \item Second, instead of systems driven by Brownian motions, we may consider systems driven by L\'evy process; some recent work can be seen in \cite{BYY17}. One could work with SPDE models driven by L\'evy processes. The recent work on switching jump diffusions \cite{CCTY} may also be adopted to the SPDE models.

\item Finally, in terms of the mathematical development, various estimates about longtime properties were
 given in average norm although the solution is in the better space $E$. Our effort in the future will be to obtain stochastic regularity of the solution by using the methods in \cite{Br,van1,van2} so that it is possible to provide estimates in the sup-norm ($\abs{\cdot}_E$). Nevertheless, some mathematical details need to carefully worked out.  The result
  in turn, will be of interests for people working on real data.
  Some other properties such as strictly positivity of the solutions and
  sharper conditions for extinction and permanence are worthy of consideration.
\end{itemize}

\subsection*{Acknowledement} We are grateful to the editors and reviewer for the evaluation. Our special
thanks go to the reviewer for the detailed comments  and suggestions
on an earlier version of the manuscript, which have much improved the paper. The research of D. Nguyen was supported in part by the National Science Foundation under grant  DMS-1853467. The research of N. Nguyen and G. Yin
was supported in part by the Army Research Office under grant W911NF-19-1-0176.

The ms is accepted for publication by the Applied Probability Trust (http://www.appliedprobability.org) in Journal of Applied Probability 57.2 (June 2020).

\end{document}